\documentclass[a4paper,12pt,reqno]{amsart}
\usepackage{amssymb,amsmath,amsfonts,amscd}
\usepackage{mathrsfs,latexsym}
\usepackage[all]{xy}
\usepackage{color}

\title{On an open problem of characterizing the birationality of 4K}
\author{Meng Chen and Yong Hu}

\address{School of Mathematical Sciences \& Shanghai Centre for Mathematical Sciences, Fudan University, Shanghai 200433, China}
\email{mchen@fudan.edu.cn}

\address{School of Mathematics, Korea Institute for Advanced Study, Seoul 02455, South Korea}
\email{yonghu11@kias.re.kr}

\thanks{The first author was supported by National Natural Science Foundation of China (\#11571076, \#11731004, \#11421061) and Program of Shanghai Subject Chief Scientist (\#16XD1400400)}

\newcommand{\bQ}{{\mathbb Q}}
\newcommand{\bP}{{\mathbb P}}

\newcommand{\rounddown}[1]{\lfloor{#1}\rfloor}

\newcommand\lrw{\longrightarrow}
\newcommand\rw{\rightarrow}

\newcommand\bZ{{\mathbb{Z}}}

\newcommand{\lsgeq}{\succcurlyeq}
\newcommand{\lsleq}{\preccurlyeq}

\newcommand{\simQ}{\sim_{\mathbb{Q}}}
\newcommand{\Mov}{\textrm{Mov}}
\newcommand{\Bs}{\textrm{Bs}}

\newtheorem{thm}{Theorem}[section]
\newtheorem{lem}[thm]{Lemma}

\theoremstyle{definition}

\newtheorem{exmp}[thm]{Example}

\theoremstyle{remark}

\begin{document}
\begin{abstract} We answer an open problem raised by Chen-Zhang in 2008 and prove that, for any minimal projective 3-fold $X$ of general type with the geometric genus $\geq 5$, the $4$-canonical map $\varphi_{4,X}$ is non-birational if and only if $X$ is birationally fibred by a pencil of $(1,2)$-surfaces (i.e. $c_1^2=1$, $p_g=2$).  The statement does not hold for those with the geometric genus $\leq 4$ according to our examples.
\end{abstract}
\maketitle
\pagestyle{myheadings}
\markboth{\hfill M. Chen and Y. Hu\hfill}{\hfill On an open problem of characterizing the birationality of 4K\hfill}
\numberwithin{equation}{section}

\section{\bf Introduction}

Throughout we work over an algebraically closed field of characteristic $0$.

In this note, a $(1,2)$-surface means a nonsingular projective surface of general type whose minimal model has the invariants: $c_1^2=1$ and $p_g=2$.

A famous theorem of Bombieri says that, for any nonsingular projective surface $S$ of general type, $\varphi_{4,S}$ is non-birational if and only if $S$ is a $(1,2)$-surface.
A direct corollary is that any nonsingular projective 3-fold of general type, admitting a pencil of $(1,2)$-surfaces, necessarily has non-birational $4$-canonical map. A very natural question (raised in Chen-Zhang \cite[6.4(1)]{CZ08}) is whether the converse is true!

A projective 3-fold $Z$ is said to be {\it (birationally) fibred by a pencil of $(1,2)$-surfaces} if $Z$ is birationally equivalent to a nonsingular projective 3-fold $Y$ which admits a fibration  $Y\rw T$ onto a smooth complete curve $T$ where the general fiber is a $(1,2)$-surface.

The purpose of this paper is to prove the following theorem which answers the above  question:

\begin{thm}\label{main}
Let $X$ be a minimal projective $3$-fold of general type with $p_g(X)\ge 5$. Then $\varphi_{4,X}$ is non-birational if and only if $X$ is birationally fibred by a pencil of $(1,2)$-surfaces.
\end{thm}

One has the following example:

\begin{exmp}\label{pg4} (see Fletcher \cite{Flet}) The general hypersurface of degree $10$:
$$X=X_{10}\subset \bP(1,1,1,1,5)$$
  is a smooth canonical 3-fold with $p_g=4$ and non-birational $4$-canonical map $\varphi_{4,X}$.  Since $X$ is a double cover onto $\bP^3$, $X$ admits no genus $2$ curve class of  canonical degree $1$. Hence $X$ admits no pencil of $(1,2)$-surfaces (by the same argument as in \cite[Example 6.3]{CZ08}).
\end{exmp}

Example \ref{pg4}, together with Example \ref{pg3} and Example \ref{pg1} in the last section, shows that the condition ``$p_g(X)\geq 5$'' in Theorem \ref{main} is sharp.
\smallskip

Throughout we use the following symbols:
\begin{itemize}
\item[$\diamond$] ``$\sim$'' denotes linear equivalence or ${\mathbb Q}$-linear equivalence (subject to the context);
\item[$\diamond$] ``$\equiv$'' denotes numerical equivalence;
\item[$\diamond$] ``$|D_1|\lsgeq |D_2|$'' (or, equivalently,  ``$|D_2|\lsleq |D_1|$'') means, for linear systems $|D_1|$ and $|D_2|$ of divisors on a variety,
$$|D_1|\supseteq|D_2|+\text{certain fixed effective divisor}.$$
\end{itemize}

\section{\bf Proof of Theorem \ref{main}}

Throughout this section, $X$ denotes a minimal projective $\bQ$FT 3-fold of general type with $p_g(X)\geq 5$. Let $K_X$ be a canonical divisor of $X$ and denote by $\mathrm{Sing}(X)$ the singular locus of $X$. Since $3$-dimensional terminal singularities are isolated, $\mathrm{Sing}(X)$ consists of only finitely many points.

\subsection{Fixed notation and the standard resolution for $\Mov|K_X|$}\label{setup}\

{}First of all, we take a resolution of singularities of $X$, say: $\alpha:X_0\lrw X$ where $X_0$ is projective. In particular, we may choose $\alpha$ such that $\alpha$ is an isomorphism over the smooth locus of $X$.  As $X$ is minimal, we have $p_g(X_0)=p_g(X)\geq 5$.
We may write $$\alpha^{*}(K_X)=M_0+Z_0',$$ where $|M_0|=\Mov|K_{X_0}|$ and $Z'_0$ is an effective $\bQ$-divisor.

By Hironaka's big theorem, we may resolve the base locus $\Bs|M_0|$ by taking
successive blowups, say:
\begin{align*}
 \beta \colon X'=X_{n+1}\stackrel{\pi_n}\rightarrow X_{n}\rightarrow \cdots \rightarrow X_{i+1} \stackrel{\pi_{i}}\rightarrow X_{i}\rightarrow \cdots \rightarrow X_1\stackrel{\pi_0}\rightarrow X_0
\end{align*}
where each $\pi_i$ is a blow-up along a nonsingular center $W_i$ ($W_i$ is contained in the base locus of the movable part $\Mov|(\pi_0\circ \pi_1 \circ \cdots \circ \pi_{i-1})^*(M_0)|$. Moreover, the morphism $\beta=\pi_n\circ \cdots \circ \pi_0$ satisfies the following properties:
\begin{enumerate}
\item The linear system $|M|=\Mov|\beta^*(M_0)|$ is base point free.
\item One may write
      \begin{eqnarray}\label{eq:blowupadjunction}
          K_{X'}&=&\beta^*(K_{X_0})+\sum_{i=0}^{n}a_i E_i,\label{eq:blowupadjunction}\\
          \beta^*(M_0)&=&M+\sum_{i=0}^{n}b_i E_i,
      \end{eqnarray}
where each $E_i$ is the strict transform of the exceptional divisor of $\pi_i$ for $0 \le i \le n$,
 $a_i$ and $b_i$ are positive integers.
\end{enumerate}

For any positive integer $m$, denote by $|M_m|$ the moving part of $|mK_{X'}|$.  By our notation, $M=M_1$.

 \begin{lem}\label{ab} (see \cite[Lemma 4.2]{Ch04}) In the above setting, the following properties hold:
\begin{itemize}
\item[(i)] For any $i$, $a_i\leq 2b_i$.

\item[(ii)] If $a_k=b_k=1$ for some $k$ with $0 \le k \le n$, then $W_k$ is a smooth curve contained in  $X_k$.

\item[(iii)]
 If $a_k=2b_k$ for some $k$ such that $0 \le k \le n$, then $W_k$ is a closed point of $X_k$.
\end{itemize}
\end{lem}

Let {$\pi=\alpha\circ\beta: X'\lrw X$} be the composition. We may write
\begin{align}\label{eq:adjunction}
K_{X'}\simQ \pi^*(K_X)+E_{\pi}, \ \ \ \ \pi^*(K_X)\simQ M+E_{\pi}',
\end{align}
where $E_{\pi}$ is an effective $\pi$-exceptional $\bQ$-divisor and $E_{\pi}'$ is an effective $\bQ$-divisor.  Let $g=\varphi_{1,X}\circ \pi$ and set $\Sigma=g(X')$.
Take the Stein factorization of $g$, say
$X^{'}\stackrel{f}\rightarrow B \stackrel{s} \rightarrow \Sigma$.
We have the following commutative diagram:
 \begin{align*}
\xymatrix{
 & X^{'} \ar[d]_{\pi} \ar[r]^{f} \ar[dr]^{g} \ar[dl]_{\beta}
                & B \ar[d]^{s} \\
X_0 \ar[r]^{\alpha} & X  \ar@{-->}[r]_{\varphi_{1,X}}
                & \Sigma          }
\end{align*}
where $B$ is a normal projective variety.

\subsection{The case of $\dim (B)=1$ and $3$}\

This is a known case since we have the following theorem:

\begin{thm}\label{d=1} (see Chen--Zhang \cite[4.2, 4.8, 4.9]{CZ08}) Let $X$ be a minimal $3$-fold of general type with $p_g(X)\geq 5$. Keep the notation in \ref{setup}.  The following statements hold:
\begin{itemize}
\item[(i)]  Assume $\dim (B)=1$. Then $\varphi_{4,X}$ is non-birational if and only if the general fiber of $f$ is a $(1,2)$-surface.
\item[(ii)] Assume $\dim(B)=3$. Then $\varphi_{4,X}$ is birational onto its image.
\end{itemize}
\end{thm}

\subsection{The case of $\dim(B)=2$}\

Let $C$ be a general fiber of $f$. The following result was proved by Chen--Zhang as well:

\begin{thm}\label{d=20}(see \cite[4.3]{CZ08}) Let $X$ be a minimal $3$-fold of general type with $p_g(X)\geq 5$. Keep the notation in \ref{setup}. Assume $\dim(B)=2$. Then $\varphi_{4,X}$ is non-birational if and only if $g(C)=2$ and $(\pi^*(K_X)\cdot C)=1$.
\end{thm}

($\ddag$) {\bf From now on, we always assume:
 $$g(C)=2\ \text{and}\ (\pi^*(K_X)\cdot C)=1.$$}

Pick a general member $S$ in $|M|$. By Chen--Zhang \cite[Theorem 2.4]{CZ16},
one has
$$|2nK_{X'}||_S\lsgeq |n(K_{X'}+S)||_S=|nK_S|$$
for any sufficiently large and divisible integer $n$. Noting that
$$2n\pi^*(K_X)\geq M_{2n}$$
and that $|n\sigma^*(K_{S_0})|$ is base point free, we have
\begin{align}\label{cri}
\pi^*(K_X)|_S\simQ \frac{1}{2}\sigma^*(K_{S_0})+H_S,
\end{align}
 where $H_S$ is an effective $\bQ$-divisor on $S$.
 We may write
 \begin{align}\label{eq:res}
 \pi^*(K_X)|_S\simQ S|_S+E_{\pi}'|_S\ \text{and}\ S|_S\equiv aC,
 \end{align}
where $a\ge p_g(X)-2\ge 3$.

 \begin{lem}\label{lem:horizontal}  Let $X$ be a minimal $3$-fold of general type with $p_g(X)\geq 5$. Keep the notation in \ref{setup}.  Assume that $\dim (B)=2$ and that $\varphi_{4,X}$ is non-birational.
Then there exists exactly one exceptional divisor $E\subset \mathrm{Supp}(E_{\pi})$ such that $(E\cdot C)=1$.
\end{lem}
\begin{proof} By Theorem \ref{d=20} and Relation \eqref{eq:res}, we have $(E_{\pi}'|_S\cdot C)=1$. By \eqref{eq:adjunction} and the assumption, we have $(E_{\pi}|_S\cdot C)=1$.

{}First we prove that the horizontal part of $\mathrm{Supp}(E_{\pi}^{'}|_S)$ is an integral curve $\Gamma$. Take a general member $K_1$ in $|K_X|$, we have $\pi^*(K_1)|_C=E_{\pi}^{'}|_C$.
It is clear that one of the following cases occurs:
\begin{enumerate}
\item[(a)] $\mathrm{Supp}(E_{\pi}'|_C)$ consists of one single point $P$ with $2P\sim K_C$;
\item[(b)] $\mathrm{Supp}(E_{\pi}'|_C)$ consists of two different points $P$ and $Q$, where $P+Q\sim K_C$.
\end{enumerate}
We will exclude the possibility of $(b)$. Otherwise, we may write $E_{\pi}'|_C=\varepsilon P+(1-\varepsilon)Q$, where $0<\varepsilon<1$.

By the argument in the proof of \cite[Proposition 4.6]{CZ08}, we know that $\Mov|4K_{X'}||_C\lsgeq |2K_C|$.
Noting that $\deg(4E_{\pi}'|_C)=4$, we see $4E_{\pi}'|_C\sim 2K_C$. Thus  $4\varepsilon$ is a positive integer. If $4\varepsilon=1$, then $4E_{\pi}^{'}|_C\sim K_C+2Q$, which implies that $2Q\sim K_C$, a contradiction. Similarly, we can conclude that $4\varepsilon\neq 3$. Thus we have $\varepsilon=\frac{1}{2}$ and  $$2P+2Q=\llcorner 5E_{\pi}^{'} |_C\lrcorner\ge {M_5}|_C,$$
for a general fiber $C$. This simply implies that $\varphi_{5,X}|_C$ is not birational and neither is $\varphi_{5,X}$, which contradicts to  \cite[Theorem 1.2(2)]{IJM}. Therefore the only possibility is case $(a)$.

Since $E_{\pi}|_C+E_{\pi}^{'}|_C\in |K_C|$ and $2P\in |K_C|$,
we have $E_{\pi}|_C=P$, which implies that the horizontal part  of $E_{\pi}|_S$ (with respect to the fibration $f|_S$) coincides with the horizontal part of $E_{\pi}^{'}|_S$ (with respect to the fibration $f|_S$). Since $\mathrm{Supp}(E_{\pi}|_C)$ consists of exactly one point for a general $C$, there exists only one exceptional divisor $E$ such that $(E\cdot C)=1$.
In particular, the coefficient of $E$ in $E_{\pi}$ (and hence in $E_{\pi}'$) is $1$.

Furthermore, for any other $\pi$-exceptional divisor $E'\neq E$, $E'|_S$ is vertical with respect to $f|_S$ for a general member $S$.
\end{proof}

By Lemma \ref{lem:horizontal},  for a general member $S\in |M|$, we may write
\begin{equation}\label{eq:hor}
E_{\pi}|_S=\Gamma+E_V, \ \ E_{\pi}'|_S=\Gamma+E_V',
\end{equation}
where $\Gamma$ is the horizontal part satisfying $(\Gamma\cdot C)=1$ for a smooth fiber $C$ contained in $S$, $E_V$ and $E_V'$ are both vertical parts with respect to $f|_S$.

\begin{lem}\label{lem:horizontal2} Let $X$ be a minimal $3$-fold of general type with $p_g(X)\geq 5$. Keep the notation in \ref{setup}.  Assume that $\dim (B)=2$ and that $\varphi_{4,X}$ is non-birational.
We have $(\pi^*(K_X)|_S\cdot \Gamma)>0$. In particular, we have $E=E_i$ for some $i$, $a_i=b_i=1$ and $\pi(E)$ is an irreducible curve on $X$.
\end{lem}

\begin{proof} It is clear that $p_g(S)=h^0((K_{X'}+S)|_S)\geq 3$,  so $S$ is not a $(1,2)$-surface and
we have $(\sigma^*(K_{S_0})\cdot C)\ge 2$ by the Hodge index theorem and the result of Bombieri \cite{Bom} that a minimal $(1,1)$ surface is simply connected (see also \cite[Lemma 2.4]{CC3} for a direct reference).

By Relation \eqref{cri}, we have $(\sigma^*(K_{S_0})\cdot C)=2$ and $(H_S\cdot C)=0$. Thus $H_S$ is composed of vertical divisors with respect to $f|_S$. Since $\Gamma$ is the section of the fibration $f|_S$, we have $(\pi^*(K_X)|_S\cdot\Gamma)\ge\frac{1}{2}(\sigma^*(K_{S_0})\cdot \Gamma)$ by \eqref{cri}. Thus it is sufficient to prove $(\sigma^*(K_{S_0})\cdot \Gamma)>0$.

Suppose $(\sigma^*(K_{S_0})\cdot \Gamma)=0$.  We consider the contraction $\sigma:S\rw S_0$ onto the minimal model $S_0$.
Since $g(C)=2$, $(\sigma^*(K_{S_0})\cdot C)=2$  and $C^2=0$, we see that all exceptional divisors of $\sigma$ are contained in special fibers of $f|_S$. Thus $C=\sigma^*(\overline{C})$ where $\overline{C}$ comes from a free pencil of genus $2$ on $S_0$.
Let $\overline{\Gamma}=\sigma_{*}(\Gamma)$. Since $(\overline{\Gamma}\cdot\overline{C})=(\Gamma\cdot C)=1$, we conclude that $\overline{\Gamma}$ is a section of fibration induced from the free pencil generated by $\overline{C}$. In particular, $\overline{\Gamma}\neq 0$. Thus $\overline{\Gamma}$ is a $(-2)$-curve on $S_0$. By the adjunction formula and \eqref{eq:hor}, we can write
\begin{align*}
K_S=(K_{X'}+S)|_S=\pi^*(K_X)|_S+S|_S+\Gamma+E_V.
\end{align*}
Considering the Zariski decomposition of the above divisor, we can write
\begin{align*}
\pi^*(K_X)|_S+S|_S+\Gamma+E_V\equiv (\pi^*(K_X)|_S+S|_S+N^{+})+N^{-},
\end{align*}
where
\begin{enumerate}
\item[(z1)] both $N^+$ and $N^-$ are effective $\bQ$-divisors and $N^++N^-=\Gamma+E_V$;
\item[(z2)] the $\bQ$-divisor $\pi^*(K_X)|_S+S|_S+N^+ $ is equal to $\sigma^*(K_{S_0})$;
\item[(z3)] $((\pi^*(K_X)|_S+S|_S+N^+) \cdot N^-)= 0$.
\end{enumerate}
Since $(\sigma^*(K_{S_0})\cdot C)=2$ and $(\pi^*(K_X)|_S\cdot C)=1$, we have $N^+=\Gamma+A$, where $A$ is an effective vertical divisor. Thus we can write
\begin{align*}
\sigma^*(K_{S_0})&=\pi^*(K_X)|_{S}+S|_S+\Gamma+A\\
&\equiv 2aC+2\Gamma+E'_V+A.
\end{align*}
Pushing forward to $S_0$, we have
\begin{align*}
K_{S_0}\equiv 2a\overline{C}+2\overline{\Gamma}+\sigma_{*}(E'_V+A),
\end{align*}
where $\sigma_{*}(E'_V+A)$ is clearly vertical.
Then we get $(K_{S_0}\cdot \overline{\Gamma})\ge 2a-4\ge 2$, which contradicts to our assumption. So our conclusion is that $(\sigma^*(K_{S_0})\cdot \Gamma)>0$.

Note that $\Gamma$ comes from the exceptional divisor $E$. Since $(\pi^*(K_X)|_S\cdot E|_S)\ge (\pi^*(K_X)\cdot\Gamma)>0$, we see that $E=E_i$ for some index $i$ by the construction of $\pi$. In particular, by Lemma {\ref{lem:horizontal2}}, we have $a_i=b_i=1$.
\end{proof}

By Lemma \ref{ab},  one sees that $E$ comes from the blow-up of a smooth curve.  Thus $E$ carries a natural fibration whose general fiber is a smooth rational curve. Denote by $l_E$ the general fiber of this fibration. We have the following observation:

\begin{lem}\label{lem:key} Under the same assumption as that of Lemma \ref{lem:horizontal2}, keep the above notation. We have $(S\cdot l_E)=1$. In particular, we have $S|_E\ge l_1+l_2$ for two distinct general elements in the same algebraic class of $l_E$  on $E$.
\end{lem}
\begin{proof}
Denote by $E_i^{i}$ the exceptional divisor of $\pi_i$ so that $E$ dominates $E_i^i$ and by $l_{E_i^{i}}$ the corresponding general ruling. We have $(E_i^{i}\cdot l_{E_i^{i}})=-1$. Denote by $\tilde{E}$ the total transform of $E_i^{i}$ on $X'$. Then we have $(\tilde{E}\cdot l_E)=-1$ by the projection formula. For any exceptional divisor $D$ not contained in $\tilde{E}$, we have $(D\cdot l_E)=0$ by the choice of $l_E$. 
By \eqref{eq:adjunction}, $(\pi^*(K_X)\cdot l_E)=0$ and our construction of $\pi$, we have $(S\cdot l_E)\le 1$. Since $f|_E$ is a birational morphism and $f$ is induced by $|S|$, we have $(S\cdot l_E)_{X'}=(S|_E\cdot l_E)_E>0$. Since $E$ is a smooth projective surface and $S|_E$ is a Cartier divisor, we have
$(S\cdot l_E)=1$.

Take two distinct general fibers $l_1$ and $l_2$ in the ruling of $E$. Since $l_E$ is a smooth rational curve, we have $h^0(l_E, S|_{l_E})=2$. Since $((S-E)\cdot C)=-1<0$, we have $h^0(X', S-E)=0$. Thus we have $h^0(E, S|_E)\ge p_g(X)\ge 5$. Consider the natural exact sequence
\begin{align*}
0\rightarrow H^0(E, S|_E-l_1-l_2)\rightarrow H^0(E, S|_E)\rightarrow H^0(l_1, S|_{l_1})\oplus H^0(l_2, S|_{l_2}).
\end{align*}
We naturally get $h^0(E, S|_E-l_1-l_2)\ge 1$, which implies that $S|_E\ge l_1+l_2$.
\end{proof}

Now we are ready to prove the main statement.

\begin{thm}\label{d212} Let $X$ be a minimal $3$-fold of general type with $p_g(X)\geq 5$. Keep the notation in \ref{setup}.  Assume that $\dim (B)=2$ and that $\varphi_{4,X}$ is non-birational. Then $X$ is birationally fibred by a pencil of $(1,2)$-surfaces.
 \end{thm}
\begin{proof} {}First of all, we note that all our above arguments remain effective if we replace $\pi$ by any further birational modification over $\pi$.

Since $E$ is birational to $B$, we may take a common smooth projective birational modification $W$ of both $B$ and $E$. Take a birational modification $\pi'\colon X''\rightarrow X'$ such that $f\circ\pi'$ factors through $W\rightarrow B$. Denote by $f''\colon X''\rightarrow W$ the corresponding fibration. The natural $\bP^1$-fibration on $E$ induces a fibration on $W$. Denote by $l_W$ the general fiber of the fibration induced from the ruling. Set $\tilde{\pi}=\pi\circ \pi'$.

Now we work on the higher model $X''$, on which we have the base point free linear system $|M''|=|\pi^*(M)|$ and the general member $S''$ has the property: $S''=\pi'^*(S)=f''^*(H)$ for a certain nef and big divisor $H$ on $W$.  By Lemma \ref{lem:key}, we have $H\ge l_{1,W}+l_{2,W}$ for two general distinct fibers on $W$ (in the same algebraic class as that of $l_W$).  Set $F''=f''^*(l_W)$ and $F_i''=f''^*(l_{i,W})$ for $i=1,2$. Clearly $F''$ induces a pencil on $X''$ and $\tilde{\pi}^*(K_X)\geq S''\geq F_1''+F_2''\equiv 2F''$.

Since $S''|_{F''}$ is moving, we have $p_g(F'')\geq 2$.  On the other hand, the canonical system $|K_{X''}|$ contains a free sub-pencil $|F_1''+F_2''|$ with a generic irreducible element $F''$, which is smooth and projective. By  \cite[Lemma 2.1]{CC3}, we have
\begin{equation}
\tilde{\pi}^*(K_X)|_{F''}\geq \frac{2}{3}\sigma''^*(K_{F_0''})
\end{equation}
where $\sigma'':F''\rw F_0''$ denotes the contraction onto the minimal model.

Denote by $C''$ a general fiber of $f''$. Pick a smooth such element $C_F''\subset F''$. Clearly we have
$$1=(\tilde{\pi}^*(K_X)|_{F''}\cdot C_F'')\geq \frac{2}{3}(\sigma''^*(K_{F_0''})\cdot C_F''),$$
which means that $(\sigma''^*(K_{F_0''})\cdot C_F'')=1$.  Hence $F_0''$ must be a $(1,2)$-surface by Bombieri (see also \cite[Lemma 2.4]{CC3} for a direct reference). We are done.
\end{proof}

Now it is clear that Theorem \ref{main} follows directly from Theorem \ref{d=1}, Theorem \ref{d=20} and Theorem \ref{d212}.  We have finished the proof of our main theorem.

\section{Examples}

It is interesting to know whether a pencil of $(1,2)$-surfaces necessarily appears in those $3$-folds of general type with $p_g\leq 4$ and with non-birational 4-canonical maps.  We provide two more examples here.

\begin{exmp}\label{pg3} Consider the general hypersurface of degree $12$ (canonical 3-fold)  $X=X_{12}\subset \bP(1,1,1,2,6)$.  One knows that $K_X^3=1$, $p_g(X)=3$ and $X$ has 2 orbifold points $\frac{1}{2}(1,-1,1)$.  It is also clear that $\varphi_{4,X}$ is non-birational.
We claim that $X$ does not admit any pencil of $(1,2)$-surfaces.

Assume, to the contrary, that $X$ admits a pencil of $(1,2)$-surfaces, say $\Lambda\subset |F_1|$ where $\dim \Lambda=1$, $F_1$ is irreducible and is of $(1,2)$-type.  We keep the notation in \ref{setup} and modify $\pi$ (for simplicity, still denoted by $\pi$), if necessary, so that $\Mov |\rounddown{\pi^*(F_1)}|$ is base point free.  Denote by $F$ the generic irreducible element of $\Mov |\rounddown{\pi^*(F_1)}|$.  By assumption, $F$ is a $(1,2)$-surface. Since $|K_X|$ is not composed of a pencil and, in fact, $\varphi_{1,X}$ induces a genus $2$ fibration (see \cite{MA}),  we see that the natural map
$$H^0(K_{X'})\lrw H^0(F, K_F)$$
is surjective for a general element $F$. In particular, $K_{X'}\geq F$ and $\pi^*(K_X)|_F\geq \Mov|K_{F}|$.  Recall that we have $\rho(X)=1$ by Dolgacgev \cite[3.2.4]{Dolg}. Then we may write
$K_X\equiv aF_1$ for some rational number $a\geq 1$. 
Since $r_X=2$, we have $2(K_X^2\cdot F_1)\in \mathbb{Z}_{>0}$. Hence $a=1$ or $2$.

{}First, we consider the case $a=1$. We have $K_X\sim F_1$ and $(K_X^2\cdot F_1)=1$.  In fact, we may take such a partial resolution $\hat{\pi}: \hat{X}\lrw X$ that $\hat{\pi}$ is a composition of blow-ups along those centers over $\Bs(\Lambda)$ and that $\Mov(\hat{\pi}^*(\Lambda))$ is free of base points. By assumption, the generic  irreducible element $\hat{F}$ in $\Mov(\hat{\pi}^*(\Lambda))$ is a nonsingular projective surface of $(1,2)$-type. Thus we may write
$$\hat{\pi}^*(F_1)=\hat{F}+E_{\hat{\pi}}',$$
$$K_{\hat{X}}=\hat{\pi}^*(K_X)+E_{\hat{\pi}},$$
where $\text{Supp}(E_{\hat{\pi}}')=\text{Supp}(E_{\hat{\pi}})$ by the construction. 
Noting that $|\hat{F}|$ is a free pencil, we have $(\hat{\pi}^*(K_X)|_{\hat{F}})^2=(K_X^2\cdot F_1)=1$.  The uniqueness of Zariski decomposition implies that $\hat{\pi}^*(K_X)|_{\hat{F}}$ is the positive part of $K_{\hat{F}}$. Thus $(\hat{\pi}^*(K_X)|_{\hat{F}}\cdot E_{\hat{\pi}}|_{\hat{F}})=0$, which also means that 
$$(K_X\cdot F_1^2)=(\hat{\pi}^*(K_X)|_{\hat{F}}\cdot E_{\hat{\pi}}'|_{\hat{F}})=0,$$
a contradiction.

We consider the case $a=2$.  Clearly we have $(K_X^2\cdot F_1)=\frac{1}{2}$.  On the other hand, we have
$\pi^*(K_X)|_S\geq \frac{1}{2}\sigma^*(K_{S_0})$ by \eqref{cri}.
Noting that $S|_F\equiv C\equiv \Mov|K_F|$, we have
$$(\pi^*(K_X)^2\cdot F)\geq (\pi^*(K_X)|_F\cdot S|_F)=(\pi^*(K_X)|_S\cdot F|_S)\geq
\frac{1}{2}(\sigma^*(K_{S_0})\cdot F|_S)\geq 1$$
by \cite[Lemma 2.4]{CC3} since $S$ is not a $(1,2)$-surface. This is also absurd. \end{exmp}

\begin{exmp}\label{pg1} Consider the general complete intersection $X=X_{6,10}\subset \bP(1,2,2,2,3,5)$, which has invariants: $K_X^3=\frac{1}{2}$, $p_g(X)=1$ and has $15$ orbifold points of type $\frac{1}{2}(1,-1,1)$. We claim that $X$ does not admit any pencil of $(1,2)$-surfaces. Assume, to the contrary, that $X$ admits a pencil $|\overline{F}|$ of $(1,2)$-surfaces where $\overline{F}$ is irreducible. 
We aim at deducing a contradiction.

Since $\rho(X)=1$ (see \cite[3.2.4]{Dolg}), we may write $K_X\equiv a\overline{F}$ for some positive rational number $a$. Noting that $(K_X^2\cdot \overline{F})\leq 1$ (since $\overline{F}$ is a $(1,2)$-surface), we have $a\geq \frac{1}{2}$. 

{}First of all, let us fix the notation. Since $P_2(X)=4$ and the bicanonical map $\varphi_{2,X}$ gives a generically finite map, we set $|\overline{S}|=\Mov|2K_X|$. Take a birational modification $\mu: \tilde{X}\lrw X$ such that the following properties hold:
\begin{itemize}
\item[(i)] $\tilde{X}$ is nonsingular and projective;
\item[(ii)] both $\Mov|2K_{\tilde{X}}|$ and $\Mov|\rounddown{\mu^*(\overline{F})}|$ are base point free.
\end{itemize}
Take general members $S\in \Mov|2K_{\tilde{X}}|$ and $F\in \Mov|\rounddown{\mu^*(\overline{F})}|$. We may write
$$\mu^*(\overline{S})\sim_{\bQ} S+E_2,$$
$$\mu^*(\overline{F})\sim_{\bQ} F+E_1,$$
where $E_1$ and $E_2$ are effective $\bQ$-divisors.
By assumption we know that $p_g(S)\geq 3$, that $\Phi_{|S|}$ is generically finite and that $F$ is a nonsingular $(1,2)$-surface.

Since $(K_X^2\cdot \overline{F})=(\pi^*(K_X)^2\cdot F)>0$ and $r_X(K_X^2\cdot \overline{F})\in \bZ$ (by the intersection theory and the fact that $X$ has isolated singularities), we see $(K_X^2\cdot \overline{F})=\frac{1}{2}$ or $1$.  In a word, either $a=\frac{1}{2}$ or $a=1$ is true.

If $a=\frac{1}{2}$, then $\overline{F}\equiv 2K_X$ and $(K_X^2\cdot \overline{F})=1$.
Since $(\mu^*(K_X)|_F)^2=(K_X^2\cdot \overline{F})=1$ and $\mu^*(K_X)|_F\leq K_F$, the uniqueness of Zariski decomposition implies that $\mu^*(K_X)|_F\sim \sigma_0^*(K_{F_0})$ where $\sigma_0:F\rw F_0$ is the contraction onto the minimal model.  The similar argument  to that in Example \ref{pg3} (the case $a=1$) shows that 
$(K_X\cdot F_1^2)=0$, a contradiction. 

If $a=1$, we have
$$2\geq (K_X\cdot \overline{S}^2)\geq (\mu^*(K_X)\cdot S^2)\geq (S|_F)^2\geq 2,$$
which implies $2K_X\equiv \overline{S}$ and $(\mu^*(K_X)\cdot S^2)=2$.  By \cite[Lemma 2.4, Corollary 2.5]{CZ16}, we have $$\mu^*(K_X)|_F\geq \frac{1}{2}\sigma_0^*(K_{F_0}).$$
Hence it follows that
$$1\leq \frac{1}{2}(\sigma_0^*(K_{F_0})\cdot S|_F)\leq (\mu^*(K_X)\cdot F\cdot S)\leq \frac{1}{2}(\mu^*(K_X)\cdot S^2)=1,$$
which implies $(\mu^*(K_X)\cdot F\cdot S)=1$.
Since, by the Hodge index theorem,
$$1=(\mu^*(K_X)|_F\cdot S|_F)\geq \sqrt{\mu^*(K_X)|_F^2\cdot S|_F^2}\geq 1,$$
one has $S|_F\equiv 2\mu^*(K_X)|_F$.  Let $C\sim S|_F$ be a general  curve. Since we have shown that $C^2=2$, $C$ must be hyperelliptic and $C|_C$ gives a $g_2^1$ of $C$.
Now we consider the linear system
$$|K_{X'}+\rounddown{4\mu^*(K_X)}|\lsgeq |\rounddown{5\mu^*(K_X)}|.$$ It is clear that, for a general member $F$ of $|F|$,
$$|K_{X'}+\rounddown{4\mu^*(K_X)}||_F\lsleq |K_F+2C|.$$
Since $|K_F+2C|$ does not give a birational map, neither do $|K_{X'}+\rounddown{4\mu^*(K_X)}||_F$,
which  contradicts  to the fact that $\varphi_{5,X}$ is birational.  The conclusion is that $X$ does not admit any pencil of $(1,2)$-surfaces.
\end{exmp}

It might be interesting to know more such examples. However the difficulty is how to prove the non-existence of a pencil of $(1,2)$-surfaces on a 3-fold.  For the case of $p_g=4$, the reader may refer to \cite{CZ16} for a complete characterization of the birationality of $\varphi_{4}$.
\medskip

\noindent{\bf Acknowledgment}. We would like to thank professor Shigeyuki Kondo and professor Keiji Oguiso for their hospitality during the first author's visit at both Nagoya University and the University of Tokyo in November of 2017.  The first author is a member of LMNS, Fudan University. The second author would like to thank professor JongHae Keum and professor Jun-Muk Hwang for their generous support during  his stay at KIAS. Special thanks  go to Chen Jiang who pointed out a wrong argument  in Example \ref{pg3} and Example \ref{pg1} in the first version of this paper.

\end{document}